\documentclass[10pt,reqno]{amsart}
\usepackage[utf8]{inputenc}
\usepackage[english]{babel}
\usepackage{amsmath, amssymb, amsthm, comment}
\usepackage{mathrsfs}
\usepackage{enumitem}
\usepackage{graphicx}
\usepackage{tikz-cd}
\usepackage{xcolor}
\numberwithin{equation}{section} 

\usepackage{diagbox} 
\usepackage{tensor}  
\usepackage[a4paper,margin=2.5cm]{geometry}

\usepackage{hyperref} 

\usepackage[alphabetic]{amsrefs} 

\theoremstyle{plain}
\newtheorem{thm}{Theorem}[section]
\newtheorem{prop}[thm]{Proposition}
\newtheorem{lemma}[thm]{Lemma}

\theoremstyle{definition}
\newtheorem{defin}[thm]{Definition}

\newtheorem{rmk}[thm]{Remark}

\allowdisplaybreaks

\newcommand{\R}{\mathbb{R}}

\def\XXint#1#2#3{{\setbox0=\hbox{$#1{#2#3}{\int}$ }
\vcenter{\hbox{$#2#3$ }}\kern-.6\wd0}}

\usepackage{scalerel,stackengine}
\stackMath
\newcommand\hhat[1]{%
\savestack{\tmpbox}{\stretchto{%
  \scaleto{%
    \scalerel*[\widthof{\ensuremath{#1}}]{\kern.1pt\mathchar"0362\kern.1pt}%
    {\rule{0ex}{\textheight}}
  }{\textheight}%
}{2.4ex}}%
\stackon[-6.9pt]{#1}{\tmpbox}%
}

\title[Fractional Vector Calculus and the Fractional Maxwell's Equations]{Fractional Vector Calculus and the Fractional Maxwell's Equations}
\author{Giovanni Covi}

\author{Ruirui Wu}

\begin{document}

\begin{abstract}
We consider a fractional variant of Maxwell’s equations, where the electric and magnetic fields are modeled as two-point fields. To formulate the system, we introduce a fractional curl operator that is compatible with the fractional divergence operator, ensuring the divergence-free condition. A key ingredient is a projection map $\Pi$ that reduces two-point fields to one-point fields. We also define a new fractional Sobolev space whose elements enjoy a fractional Helmholtz decomposition and observe that the projection $\Pi$ is a bijection in this space, which allows us to reformulate the problem entirely in terms of one-point fields. We then prove the well-posedness of the equations in one-point fields in weighted fractional Sobolev spaces, and deduce a corresponding well-posedness result for the two-points fractional Maxwell system. This constitutes a first necessary step towards the resolution of a scattering inverse problem for the fractional Maxwell's equations, which will be the topic of future work.
\end{abstract}

\maketitle

\section{Introduction}
We consider a fractional analogue of the classical Maxwell equations
\begin{equation*}
    \begin{cases}
        -\frac{\partial \mathbf D}{\partial t} + \nabla\times \mathbf H = \mathbf J, \\
        \frac{\partial \mathbf B}{\partial t} + \nabla\times \mathbf E = 0, \\
        \mathcal \nabla \cdot \mathbf B=0, \\
        \nabla \cdot \mathbf D=\rho.
    \end{cases}
\end{equation*}
The vector fields $\mathbf E$ and $\mathbf H$ describe the {electric and magnetic fields} in space. 
The material parameters $\varepsilon$ and $\mu$ denote the {electric permittivity and magnetic permeability}. 
The function $\mathbf J$ represents an impressed {current density}, while $\rho$ denotes the {electric charge density}.

For nonlocal electrodynamics, \cite{Tarasov08} considered replacing the local curl operator with fractional Caputo derivatives. In our paper, we define a fractional analogue of the local curl that relates it to the fractional Laplacian $(-\Delta)^s$. In particular, we use the singular integral definition of the fractional Laplacian to define a fractional curl operator, which appears as an integral. We denote this fractional curl operator by $\widetilde{\mathcal{C}}$. We then consider the equation system
\begin{equation}
    \begin{cases}
        \frac{\partial \mathbf (\varepsilon \mathbf E)}{\partial t} - \widetilde{\mathcal{C}}\mathbf H = -\mathbf J, \\
        \frac{\partial( \mu \mathbf H)}{\partial t} + \widetilde{\mathcal C}\mathbf E = 0, \\
        \mathcal D (\mu \mathbf H)=0, \\
        \mathcal D (\varepsilon \mathbf E)=\rho.
    \end{cases}
\end{equation}
Here $\mathbf{E}$, $\mathbf{H}$ and $\mathbf J$ are two-point vector fields, that is, they map $\mathbb R^{2n}\ni(x,y)\mapsto v(x,y) \in\mathbb R^n$. The charge density $\rho$, the electric permittivity $\varepsilon$ and the magnetic permeability $\mu$ are scalar (one-point) fields. The fractional divergence $\mathcal D$ maps two-point vector fields to one-point scalar fields, and the two-point fractional curl $\widetilde{\mathcal{C}}$ maps two-point vector fields to two-point vector fields. We will define the fractional operators $\mathcal D, \widetilde{\mathcal C}$, as well as many other related operators, in Section \ref{Prelim} and \ref{Section:fractional Helmholtz decomposition}. These are nonlocal operators which share many of the familiar properties of classical differential operators, such as $\nabla\cdot(\nabla\times v)= 0$ and $\nabla\times (\nabla \varphi) =0$. They are derived from the nonlocal vector calculus of \cite{DGLZ13}, see also \cite{Covi20} and \cite{Covi20b}.

Like in the classical case, we consider the obstacle scattering problem for electromagnetic waves.
Let $\Omega \subset \mathbb{R}^3$ be a bounded domain with sufficiently smooth
boundary representing a perfectly conducting obstacle. The electric field
$E$ is decomposed as
\[
E = E_i + E_s,
\]
where $E_i$ is an incident field solving the homogeneous background equation
in $\mathbb{R}^3$, and $E_s$  is the scattered field generated by the scattering inside $\Omega$. We want to prove well-posedness for the direct problem consisting in finding the scattered wave $E_s$ given the incident wave $E_i$. In order to prove this, we will assume that the incident wave $E_i$ solving the homogeneous equation takes a particular form, reminiscent of the one assumed in the classical case. We will also define a specific concept of solution (see Definition \ref{def:well-posedness-2-points}).

\subsection{Main results}
We will use the adjoint fractional curl operator $\mathcal{C}^*$ in \cite{DGLZ13}, which maps one-point fields to two-point fields. By additionally defining a projection map $\Pi$ that maps two-point fields to one-point fields, we define
$$
\widetilde{\mathcal C}:= \mathcal{C}^*\Pi  :  H^r_{fHd}(\R^{3}\times \R^3) \to H^{r-s}_{fHd}(\R^{3}\times \R^3 ). 
$$
where $H^r_{fHd}(\R^{3}\times \R^3 )$ and $H^{r-s}_{fHd}(\R^{3}\times \R^3 )$ are spaces of fields with fractioinal Helmholz decomposition, which will be introduced in Section \ref{Section:fractional Helmholtz decomposition}. We shall notice 
$$
\widetilde{\mathcal C}=\Pi^{-1}(\Pi \mathcal{C}^*)\Pi, 
$$
where $\Pi \mathcal{C}^*$ maps one-point fields to one-point fields. This will allow us to reduce the complicated well-posedness problem for two-point fields to a relatively more accessible well-posedness problem for one-point fields.

We consider the time-harmonic system, where 
$$
\mathbf{E}(x,y,t) = e^{-i\omega t}\varepsilon_0^{-1/2}E(x,y), \qquad \mathbf{H}(x,y,t) = e^{-i\omega t}\mu_0^{-1/2}H(x,y).
$$
where $\varepsilon_0$, $\mu_0$ are constant background electric permittivity and magnetic permeability.

Now let $\varepsilon_r:=\varepsilon/\varepsilon_0$, $\mu_r:=\mu/\mu_0$, and assume $\mu_r\equiv 1$ in $\R^3$. We recall that $\varepsilon_r=\varepsilon_r(x)$ is a one-point scalar field, i.e. it depends only on $x$. {We also make the following assumption on $\varepsilon_r$:
\[
\varepsilon_r \in C^\infty(\mathbb{R}^3),  \quad \varepsilon_r=1 \quad \text{outside a given bounded domain} \quad \Omega \subset \R^3, \]
\[
\text{and } \varepsilon_r > 0 \text{ has a positive lower bound.}
\tag{H}
\]
}

 In the obstacle scattering problem, $E=E_i+E_s$, 
 $E_i$ is a given \emph{incident field} solving the background (homogeneous) equation
  \[
  \widetilde C \widetilde C E_i - k^2 E_i = 0
  \qquad \text{in } \mathbb{R}^{3}\times \R^3,
  \]
   and $E_s$ is the \emph{scattered field}, generated by the inhomogeneity in
  $\varepsilon_r$, solving 
\begin{equation}\label{intro_two_pt} 
\widetilde C \widetilde C E_s - k^2 \varepsilon_r E_s
=
k^2(\varepsilon_r - 1)E_i \qquad \text{in } \mathbb{R}^{3}\times \R^3.
\end{equation}

We will show that
$$E_i = \mathcal{C}^* \nabla \times (pe^{ik^{1/s}x\cdot d})$$
solves the equation $\widetilde{\mathcal C} \widetilde{\mathcal C}  E_i-k^2  E_i=0$, where $p\in \R^n$ represents the polarization of the propagating wave, and $d\in \R^n$ represents the direction of propagation of the wave. This is an analogue to the incident wave $E_i = \nabla\times\nabla\times (pe^{ikx\cdot d})$ in the obstacle scattering for the classical Maxwell equations.

Applying the projection $\Pi$ to the above equation for $E_s$, and denoting $\widetilde{E}_i:=\Pi E_i$, $\widetilde{E}_s:=\Pi E_s$, we have 
$$
\Pi \mathcal{C}^* \Pi \mathcal{C}^* \widetilde{E}_s-k^2(\varepsilon_r \widetilde{E}_s+(\varepsilon_r-1)\widetilde{E}_i)=0, \quad in \quad \R^{3},
$$
which we will rewrite as
\begin{align}\label{eq_tildeE_s}
    I_{2-2s}\nabla \times \nabla \times \widetilde{E}_s-k^2\varepsilon_r \widetilde{E}_s=k^2(\varepsilon_r-1)\widetilde{E}_i,\quad in \quad \R^3.
\end{align}
Here $I_\alpha$ is the Riesz potential. Observe that this is an equation in $\mathbb R^n$, and that its unknown is the one-point field $\widetilde E_s$. For equation \eqref{eq_tildeE_s} we have obtained the following  well-posedness theorem. It holds in the weighted fractional Sobolev spaces $H^s_\delta$, which we will introduce in Section \ref{Prelim}.
\begin{thm}\label{theorem1.1}
Let $\delta>0$, $s\in [\frac{1}{2},1)$, {and $\varepsilon_r$ satisfies assumption (H)}.
    If for $\widetilde{E}_i=0$, \eqref{eq_tildeE_s} has a unique solution $\widetilde{E}_s=0$ in $H^s_{\delta}$, then
    for any fixed $p,d\in \R^n$, and $\widetilde{E}_i=\Pi E_i =\Pi \mathcal{C}^* \nabla \times (pe^{ik^{1/s}x\cdot d}) $, there is a unique solution $\widetilde{E}_s \in H^s_{\delta}$ to equation \eqref{eq_tildeE_s}.
\end{thm}
To prove the theorem, we first rewrite equation \eqref{eq_tildeE_s} as
\begin{align}\label{intro_EsEi}
    (-\Delta)^s\widetilde E_s +P_{2s-1}\widetilde E_s-k^2\varepsilon_r \widetilde{E}_s=k^2(\varepsilon_r-1)\widetilde{E}_i- P_{2s-1} \widetilde E_i \qquad \mbox{in } \R^3.
\end{align}
where $P_{2s-1}$ is a pseudo differential operator of order $2s-1$. Then we will use the Lax-Milgram theorem and the spectral theorem to show the well-posedness of equation \eqref{intro_EsEi}, as in Proposition \ref{well-posedness}. To use the spectral theorem, we need a Rellich lemma on an unbounded domain to show the compactness of the solution map, see Lemma \ref{Rellich_1}. This lemma is also where we need to use the weighted space.

Next, we define that $E_s \in H_{fHd}^{s}(\R^{3}\times \R^3)$ is a solution to 
\begin{equation*}
(\widetilde {\mathcal{C}} \widetilde {\mathcal{C}}  -k^2 \varepsilon_r) E_s = k^2 (\varepsilon_r-1)E_i \qquad \mbox{in } \R^3\times \R^3.
\end{equation*}
if and only if $\Pi E_s \in H_\delta^s(\R^3)$ and 
$\Pi E_s$ solves
\begin{align*}
    (-\Delta)^sv +P_{2s-1}v-k^2\varepsilon_r v=k^2(\varepsilon_r-1)\widetilde{E}_i- P_{2s-1} \widetilde E_i \qquad \mbox{in } \R^3.
\end{align*}
With this definition, we have the well-posedness result for the two-point fields.
\begin{thm}\label{theorem1.2}
Let $\delta>0$, $s\in [\frac{1}{2},1)$, {and $\varepsilon_r$ satisfies assumption (H)}.
    If for ${E}_i=0$, \eqref{intro_two_pt} has a unique solution ${E}_s=0$ in $H^s_{fHd}(\R^{3}\times \R^3)$, then
    for any fixed $p,d\in \R^n$, and ${E}_i= \mathcal{C}^* \nabla \times (pe^{ik^{1/s}x\cdot d}) $, there is a unique solution ${E}_s \in H^s_{fHd}(\R^{3}\times \R^3)$ to equation \eqref{intro_two_pt}. 
\end{thm}

\subsection{Motivation and connection to the literature}
{The fractional Calder\'on problem has been extensively studied in recent years, primarily in the setting of the fractional Schr\"odinger equation. The inverse problem was first solved for potentials $q \in L^\infty$ in \cite{GSU20}, and later extended to lower regularity classes in \cite{RulandSalo2020}. Uniqueness and reconstruction were shown to hold even from a single measurement in \cite{GhoshRulandSaloUhlmann2020}. Stability properties have also been investigated, including exponential-type estimates and their optimality \cite{RulandSalo2018,RulandSalo2020,Ruland2021}, as well as Lipschitz stability in finite-dimensional settings \cite{RulandSincich2019}.

The conductivity formulation of the fractional Calder\'on problem was introduced in \cite{Covi20}, and discussed in \cite{Covi20b}, \cite{Li2020}, \cite{Li2021} in a magnetic field, which defined the fractional gradient and divergence operators, building on the nonlocal vector calculus developed in \cite{DGLZ12}, \cite{DGLZ13} for nonlocal volume-constrained problems. It is natural to ask whether there also exists a fractional analogue of the curl operator, which we will use in the formulation of fractional Maxwell equations. }

Discussion on nonlocal Maxwell equations can be found in \cite{Tarasov08}, which has potential applications in fractional nonlocal electrodynamics with power–law type nonlocality, or
describing nonlocal properties in classical dynamics.

For obstacle scattering for classical Maxwell equations, Colton and Kress \cite{ColtonKress1998}  formulated the time-harmonic Maxwell scattering problem using boundary integral equations and established well-posedness via Fredholm theory together with the Silver–Müller radiation condition. A complementary Sobolev-space and variational treatment can be found in \cite{Monk03} by Monk, where exterior boundary value problems for the time-harmonic Maxwell system are formulated in Sobolev spaces adapted to the curl operator, and well-posedness of the obstacle scattering problem is established.

\subsection{Organization of the paper}
{In Section \ref{Prelim}, we will introduce some function spaces and fractional operators that are used in the paper. In Section \ref{Section:fractional Helmholtz decomposition}, we define the function spaces with fractional Helmholtz decomposition, and introduce the projection operator $\Pi$. Section \ref{Section: fractional Maxwell equations} defines the fractional curl operator and the fractional Maxwell equations; we also formulate the scattering problem and define solutions in two-point fields. Section \ref{section: well-posedness} first shows the well-posedness for the scattering problem for one-point fields, and that for two-point fields follows.}
\bigskip

\section{Preliminaries}\label{Prelim}
In this section, we introduce the function spaces and fractional operators which we will use in the paper. Even though many of the following definitions hold more in general, we shall always assume in the following that $n=3$.

\subsection{Function spaces} 
First, we introduce the Sobolev spaces on $\R^n$. Let $s\in (0,1)$, and define $\langle\xi\rangle:=(1+|\xi|^2)^{1/2}$ for $\xi\in\mathbb R^n$. We let $H^s= H^s(\mathbb R^n)$ be the fractional Sobolev space (in the sense of Bessel potentials) defined as
\[
H^s(\mathbb{R}^n)
:= \left\{ u\in \mathcal{S}'(\mathbb{R}^n)\;:\;
\|u\|_{H^s(\mathbb{R}^n)}<\infty \right\},
\]
where the fractional Sobolev norm is
$$\|u\|_{H^s(\mathbb{R}^n)}:= \int_{\mathbb{R}^n} \langle \xi\rangle^{2s}
|\widehat{u}(\xi)|^2\, d\xi.$$

Following McLean \cite{Mc00}, we also introduce the following fractional Sobolev spaces on an open set $U\subset \R^n$:
\[
H^s(U)
:= \{\, u|_U \;:\; u\in H^s(\mathbb{R}^n)\,\},
\qquad 
H_0^s(U)
:= \mbox{closure of } C_c^\infty(U) \mbox{ in } H^s(\mathbb{R}^n),
\]
and we associate to the first one the norm
$$\|u\|_{H^s(U)}
:= \inf\{\|v\|_{H^s(\mathbb{R}^n)}:\ v|_U=u\}.$$
We also let
$$
H^{-s}(U)
:= (H_0^s(U))^*.$$

Most of our results will involve the fractional Sobolev spaces discussed above. However, in order to study the well-posedness of nonlocal partial differential equations defined on the whole space $\mathbb R^n$, we need to introduce weighted fractional Sobolev spaces. Let $s\in(0,1)$ and $\delta\geq 0$. The weighted fractional Sobolev space $H^s_\delta$ is defined as
$$ H^s_\delta := \{ u \in L^2 : \|u\|_{H^s_\delta} < \infty  \}, $$
where

$$ \|u\|^2_{H^s_\delta} := \|u\|^2_{L^2_\delta} + \|(-\Delta)^{s/2}u\|^2_{L^2_\delta},  $$
and $$ \|u\|_{L^2_\delta}^2 := \int_{\mathbb R^n} \langle x \rangle^{2\delta}|u(x)|^2 \,dx. $$
It is immediate to observe that $H^s_0 =H^s$ holds for all $s\in(0,1)$, that is, the case $\delta=0$ reduces to the unweighted fractional Sobolev spaces. Moreover, $H^0_\delta = L^2_\delta$ holds for all $\delta\geq 0$. 

Our definition of weighted fractional Sobolev spaces is based on the fractional Sobolev spaces of the Bessel kind. For the sake of completeness, we also mention that one could alternatively define the weighted fractional Sobolev spaces in the Slobodeckij sense. This is similar to the usual Gagliardo seminorm:
$$ \|u\|_{W^s_\delta}^2 := \int_{\mathbb R^n} |u(x)|^2\langle x \rangle^{2\delta}dx + \int_{\mathbb R^n}\int_{\mathbb R^n}\frac{|u(x)-u(y)|^2}{|x-y|^{n+2s}} \langle x \rangle^{\delta}\langle y \rangle^{\delta}.
$$
However, we will not make use of the above space in the present work.

\bigskip

\subsection{Fractional operators} 
We now introduce the fractional operators appearing in the present paper. Let $s\in (0,1)$ and $u\in\mathcal S$, the space of Schwartz functions in $\mathbb R^n$. The fractional Laplacian $(-\Delta)^s$ can be defined via the Fourier transform $\mathcal F$ as 
\[
(-\Delta)^s u
:= \mathcal{F}^{-1}\!\left(|\xi|^{2s}\widehat{u}(\xi)\right).
\]
There exist many other equivalent definitions of the fractional Laplacian, as shown in \cite{Kwa17}. For example, it can be also defined as the following singular integral
\[
(-\Delta)^s u(x)
= C_{n,s}\,\mathrm{p.v.}\!\int_{\mathbb{R}^n}
\frac{u(x)-u(y)}{|x-y|^{n+2s}}\,dy,
\]
where the involved constant is given by $C_{n,s} := \frac{2^{s}\,\Gamma\!\left(\frac{n+s}{2}\right)}{\pi^{n/2}\,|\Gamma\!\left(-\frac{s}{2}\right)|}$. The fractional Laplacian can be extended to act on fractional Sobolev spaces \cite{GSU20}. For all $r\in\mathbb R$, we have the mapping property
\[
(-\Delta)^s : H^r\to H^{r-2s}.
\]
Moreover, for all $\delta\geq 0$ the fractional Laplacian acts as
$$(-\Delta)^{s/2}: H^s_\delta \rightarrow L^2_\delta$$
on weighted fractional Sobolev spaces, as it immediately follows from the definition of the $H^s_\delta$ norm. If $-\,\frac{n}{2} < s < 0$, the fractional Laplacian $(-\Delta)^s$ is
the Riesz potential
\[
(-\Delta)^s u
= I_{2|s|}u
= \frac{C_{n,s}}{|\cdot|^{\,n-2|s|}} * u.
\]
It can be extended as a bounded map
\[
(-\Delta)^s :
L^p(\mathbb{R}^n)
\longrightarrow
L^{\frac{np}{\,n-2|s|p\,}}(\mathbb{R}^n),
\qquad
1 < p < \frac{n}{2|s|}.
\]

Following \cite{DGLZ12, DGLZ13}, we now introduce many fractional operators defined on two points. Given the two-point vector function
$\nu(x,y):\mathbb{R}^n\times\mathbb{R}^n\to\mathbb{R}^k$
and the antisymmetric two-point vector function
$\alpha(x,y):\mathbb{R}^n\times\mathbb{R}^n\to\mathbb{R}^k$,
the \emph{nonlocal point divergence operator} $\mathcal{D}$
 is defined by
\begin{equation}
\mathcal{D}(\nu)(x)
:= \int_{\mathbb{R}^n} (\nu+\nu')\cdot \alpha \, dy,
\qquad x\in\mathbb{R}^n,
\end{equation}
where $\mathcal{D}(\nu):\mathbb{R}^n\to\mathbb{R}$, and $\nu'(x,y) = \nu(y,x)$.

Similarly, 
given the two-point vector function
$\eta(x,y):\mathbb{R}^n\times\mathbb{R}^n\to\mathbb{R}^k$
and the antisymmetric two-point vector function
$\beta(x,y):\mathbb{R}^n\times\mathbb{R}^n\to\mathbb{R}^k$,
the \emph{nonlocal point gradient operator} $\mathcal{G}$
is defined by
\begin{equation}
\mathcal{G}(\nu)(x)
:= \int_{\mathbb{R}^n} (\eta+\eta')\,\beta \, dy,
\qquad x\in\mathbb{R}^n,
\end{equation}
where $\mathcal{G}(\eta):\mathbb{R}^n\to\mathbb{R}^k$.

Given the two-point vector function
$\mu(x,y):\mathbb{R}^n\times\mathbb{R}^n\to\mathbb{R}^3$
and the antisymmetric vector two-point function
$\gamma(x,y):\mathbb{R}^n\times\mathbb{R}^n\to\mathbb{R}^3$,
the \emph{nonlocal point curl operator} $\mathcal{C}$
is defined by
\begin{equation}
\mathcal{C}(\mu)(x)
:= \int_{\mathbb{R}^n} \gamma \times (\mu+\mu') \, dy,
\qquad x\in\mathbb{R}^n,
\end{equation}
where $\mathcal{C}(\mu):\mathbb{R}^n\to\mathbb{R}^3$.

By direct computation of the adjoint operator, we have the following. Given the scalar function $u:\R^n \to \R$, the adjoint of $\mathcal{D}$ 
is given by
\begin{equation}
\mathcal{D}^*(u)(x,y)
= -(u'-u)\,\alpha,
\qquad x,y\in\mathbb{R}^n,
\end{equation}
where $\mathcal{D}^*(u):\mathbb{R}^n\times\mathbb{R}^n\to\mathbb{R}^k$.

Given the vector point function
$\mathbf{v}(x):\mathbb{R}^n\to\mathbb{R}^k$,
the adjoint of $\mathcal{G}$ is given by
\begin{equation}
\mathcal{G}^*(\mathbf{v})(x,y)
= -(\mathbf{v}'-\mathbf{v})\cdot \beta,
\qquad x,y\in\mathbb{R}^n,
\end{equation}
where $\mathcal{G}^*(\mathbf{v}):\mathbb{R}^n\times\mathbb{R}^n\to\mathbb{R}$.

Given the vector point function
$\mathbf{w}(x):\mathbb{R}^n\to\mathbb{R}^3$,
the adjoint of $\mathcal{C}$ is given by
\begin{equation}
\mathcal{C}^*(\mathbf{w})(x,y)
= \gamma \times (\mathbf{w}'-\mathbf{w}),
\qquad x,y\in\mathbb{R}^n,
\end{equation}
where $\mathcal{C}^*(\mathbf{w}):\mathbb{R}^n\times\mathbb{R}^n\to\mathbb{R}^3$.

Observe that, unlike the nonlocal divergence, gradient and their adjoints, the nonlocal curl and its adjoint are defined on functions with values in $\mathbb R^3$. In our applications, we will in general focus on the physically relevant case where the dimension is $n = 3$. Moreover, we will assume that the antisymmetric functions $\alpha, \beta, \gamma$ coincide, and are given by the vector function
 $$
\alpha(x,y) := \frac{C_{n,s}^{1/2}}{\sqrt{2}} \frac{y-x}{|y-x|^{\frac{n}{2}+s+1}},
$$
where $s\in (0,1)$ and $C_{n,s}$ is the same constant as in the definition of the fractional Laplacian. The choice of this $\alpha$ is motivated by the definition of the fractional Laplacian as a singular integral. As in \cite{Covi20}, this allows us to define the \emph{fractional gradient}  $\nabla^s :  H^{s}(\mathbb{R}^n) \to L^2(\mathbb{R}^{2n})$ as
\begin{equation*}
\nabla^s u(x,y)
:= - \frac{C_{n,s}^{1/2}}{\sqrt{2}} \,
\frac{u(y)-u(x)}{|y-x|^{n/2+s+1}} \, (y-x),
\end{equation*}
and its adjoint,
the \emph{fractional divergence}
$(\nabla \cdot)^s : L^2(\mathbb{R}^{2n}) \to H^{-s}(\mathbb{R}^n)$, by
\begin{equation}
\langle (\nabla \cdot)^s v, u \rangle_{L^2(\mathbb{R}^n)}
:=
\langle v, \nabla^s u \rangle_{L^2(\mathbb{R}^{2n})}.
\end{equation}
Then Lemma 2.1 in \cite{Covi20} shows that
\[
(\nabla \cdot)^s \bigl(\nabla^s u\bigr)(x)
=
(-\Delta)^s u(x).
\]

Using the symbols defined above, we have from \cite{Covi20b} that $\mathcal{D}^*: H^s(\mathbb{R}^n) \rightarrow L^2(\mathbb{R}^{2n})$. We observe that the estimate holds true also for $s\in\mathbb R^+\setminus \mathbb Z$, where for the definition of the higher order fractional gradient we refer to \cite{CMR20}. A similar mapping property is true also for $\mathcal{C}^*$, since
\begin{align*}
    \left\|\mathcal{C}^* f\right\|_{L^2\left(\mathbb{R}^{2 n}\right)}^2 & =\int_{\mathbb{R}^{2 n}}|\alpha \times(f(y)-f(x))|^2 d x d y \\ & \leq \int_{\mathbb{R}^{2n}}|\alpha|^2|f(y)-f(x)|^2 d x d y=\left\|\mathcal{D}^* f\right\|_{L^2(\R^{2n})}^2 .
\end{align*}
Moreover, by Lemma 2.6, Lemma 2.7 in \cite{Covi20b}, the mapping property $$\mathcal{D}^* : H^r(\R^n)\to H^{r-s}(\R^{2n})$$ holds for all $r\leq s$. Adapting the proofs in \cite{Covi20b}, we can show that the same holds for $\mathcal{C}^*$ as well. In fact, we have the following two lemmas:

\begin{lemma}\label{first-lemma-curl}
    Let $u\in C^\infty_c(\mathbb R^n)$. Then there exists a constant $k_{n,s}>0$ such that
    $$ \mathcal F (\mathcal C^* u)(\xi,\eta) = k_{n,s}\bigg( \frac{\xi}{|\xi|^{n/2+s-1}} + \frac{\eta}{|\eta|^{n/2+s-1}}  \bigg)\times\mathcal Fu(\xi+\eta).$$
\end{lemma}

\begin{lemma}\label{second-lemma-curl}
    Let $s\in(0,1)$ and $r\in(0,2s)$. The fractional curl $\mathcal C^*$ extends as a bounded map 
    $$ \mathcal C^* : H^{r}(\R^n)\rightarrow \langle D_x+D_y \rangle^{r-s}L^2(\R^{2n}), $$
    and if $r\leq s$ then 
    $$\mathcal C^* : H^r(\R^n)\rightarrow H^{r-s}(\R^{2n}).$$
\end{lemma}

For the convenience of the reader, we prove both lemmas here. The proof of the second one actually requires an estimate that differs slightly from the one appearing in the corresponding proof for the fractional gradient operator $\mathcal D^*$.

\begin{proof}[Proof of Lemma \ref{first-lemma-curl}]
    Because $u\in C^\infty_c(\mathbb R^n)$, we already know that $\mathcal C^*u\in L^2(\mathbb R^{2n})$, and thus we can take the Fourier transform $\mathcal F(\mathcal C^* u)$. By the pointwise definition of $\mathcal C^*$ and change of variables, we have
    \begin{align*}
        \mathcal F(\mathcal C^* u)(\xi,\eta) & = \int_{\mathbb R^{n}} \int_{\mathbb R^{n}} e^{-i(x\cdot\xi+y\cdot\eta)}\alpha(x,y)\times (u(y)-u(x)) dxdy 
        \\ & \approx
        \int_{\mathbb R^{n}} \frac{e^{-i z\cdot\eta} z}{|z|^{\frac{n}{2}+s+1}}\times\int_{\mathbb R^{n}} e^{-ix\cdot(\xi+\eta)} (u(x+z)-u(x)) dxdz
        \\ & =
        \int_{\mathbb R^{n}}e^{-i z\cdot\eta}( e^{iz\cdot(\xi+\eta)}-1 ) \frac{ z}{|z|^{\frac{n}{2}+s+1}}dz\times \mathcal F u(\xi+\eta),
    \end{align*}
    where the implied constant depends only on $n,s$. Now the lemma follows by solving the remaining integral in the same way as in \cite{Covi20b}.
\end{proof}

\begin{proof}[Proof of Lemma \ref{second-lemma-curl}]
    Let us start with $u\in C^\infty_c(\mathbb R^n)$. Then by Lemma \ref{first-lemma-curl}
    \begin{align*}
        \mathcal F(\langle D_x+D_y\rangle^{2(r-s)}\mathcal C^*u) & = k_{n,s}(1+|\xi+\eta|^2)^{r-s}\bigg( \frac{\xi}{|\xi|^{n/2+s-1}} + \frac{\eta}{|\eta|^{n/2+s-1}}  \bigg)\times\mathcal Fu(\xi+\eta)
        \\ & = 
        k_{n,s}\bigg( \frac{\xi}{|\xi|^{n/2+s-1}} + \frac{\eta}{|\eta|^{n/2+s-1}}  \bigg) \times \mathcal F( \langle D_x\rangle^{2(r-s)}u )(\xi+\eta)
        \\ & =
        \mathcal F(\mathcal C^*(\langle D_x\rangle^{2(r-s)}u)).
    \end{align*}
    Thus by passing to the Fourier side we see that
    \begin{align*}
        \|\mathcal C^*u\|^2_{\langle D_x+D_y\rangle^{r-s}L^2} & = \langle \mathcal C^*u , \langle D_x+D_y\rangle^{2(r-s)}\mathcal C^*u \rangle_{L^2}
        =
        \langle \mathcal C^*u , \mathcal C^*(\langle D_x\rangle^{2(r-s)}u) \rangle_{L^2}.
    \end{align*}
Now the proof of \cite[Lemma 2.7]{Covi20b} uses the equality $\mathcal D\mathcal D^* = (-\Delta)^s$ in order to obtain the desired inequality, but in this case we have instead $\mathcal C\mathcal C^*$. Let us define $v:= \langle D_x\rangle^{2(r-s)}u$. Then by Cauchy-Schwartz
\begin{align*}
    \langle \mathcal C^* u,\mathcal C^* v\rangle_{L^2} & =\langle \alpha(x,y)\times (u(y)-u(x)),\alpha(x,y)\times (v(y)-v(x))\rangle_{L^2}
    \\ & \leq 
    \int_{\mathbb R^{2n}} |\alpha(x,y)|^2|u(y)-u(x)||v(y)-v(x)|dxdy
    \\ & \approx 
    \int_{\mathbb R^{2n}} \frac{|u(y)-u(x)||v(y)-v(x)|}{|x-y|^{n+2s}}dxdy
    \\ & \leq
    \bigg(\int_{\mathbb R^{2n}} \frac{|u(y)-u(x)|^2}{|x-y|^{n+2r}}dxdy\bigg)^{1/2}\bigg(\int_{\mathbb R^{2n}} \frac{|v(y)-v(x)|^2}{|x-y|^{n+4s-2r}}dxdy\bigg)^{1/2}.
\end{align*}

Let us assume first that $r, 2s-r \not\in\mathbb Z$. Then by the mapping properties of the fractional gradient we can estimate
\begin{align*} 
\langle \mathcal C^* u,\mathcal C^* v\rangle_{L^2} 
     & \lesssim
    \|\nabla^ru\|_{L^2} \|\nabla^{2s-r}v\|_{L^2}\leq
    \|u\|_{H^r}\|v\|_{H^{2s-r}}.
\end{align*}
If instead $r\in\mathbb Z$, then by the assumptions on $r$ and $s$ it can only be $r=1$. In this case, by dominated convergence we can compute
\begin{align*}
    \bigg(\int_{\mathbb R^{2n}} \frac{|u(y)-u(x)|^2}{|x-y|^{n+2}}dxdy\bigg)^{1/2} & = \lim_{t\rightarrow 1^{-}} \bigg(\int_{\mathbb R^{2n}} \frac{|u(y)-u(x)|^2}{|x-y|^{n+2t}}dxdy\bigg)^{1/2} 
    \\ & \approx
    \lim_{t\rightarrow 1^{-}} \|\nabla^tu\|_{L^2}
    \\ & \leq
    \lim_{t\rightarrow 1^{-}} \|u\|_{H^t},
\end{align*}
and because we have $t<1$, the last term can be estimated by $\|u\|_{H^1}$. We proceed in a similar way if $2s-r\in\mathbb Z$, because also in this case it can only be $2s-r=1$. Observe that it can never happen that both $r,2s-2\in\mathbb Z$. In any case, we obtain
\begin{align*} 
\langle \mathcal C^* u,\mathcal C^* v\rangle_{L^2} 
     & \lesssim
    \|u\|_{H^r}\|v\|_{H^{2s-r}}.
\end{align*}
Now because $\|v\|_{H^{2s-r}}=\|\langle D_x\rangle^{2(r-s)}u\|_{H^{2s-r}}\leq \|u\|_{H^r}$, we obtain the estimate
$$ \|\mathcal C^*u\|_{\langle D_x+D_y\rangle^{r-s}L^2}\lesssim \|u\|_{H^r}, $$
which proves the first half of the statement after an argument by density. The second half now follows from the observation that, if $r\leq s$, then
$$ \langle D_x+D_y\rangle^{r-s}L^2(\mathbb R^{2n}) \subseteq H^{r-s}(\mathbb R^{2n}).$$
\end{proof}

\section{The fractional Helmholtz decomposition}\label{Section:fractional Helmholtz decomposition}

\subsection{Sobolev space with fractional Helmholtz decomposition}
We now define a new Sobolev space, which contains all those functions which have a \emph{fractional Helmholtz decomposition}. Given $r\in\mathbb R$ and $s\in(0,1)$, we let
\begin{align*}
    H^r_{fHd}(\R^{2n}):= 
    \{\mathcal{D}^*\varphi + \mathcal{C}^*A: \varphi \in H^{r+s}(\R^n), A\in H^{r+s}(\R^n), \nabla\cdot A=0\} 
\end{align*}

By the mapping properties of $\mathcal D^*$ and $\mathcal C^*$, it immediately holds that $ H^r_{fHd}(\R^{2n}) \subseteq L^2(\R^{2n})$ if $r \geq 0$, and $ H^r_{fHd}(\R^{2n}) \subseteq H^r(\R^{2n})$ if $r < 0$. 

The functions belonging to $H^r_{fHd}$ can thus be written as sums of a fractional gradient and a fractional curl. This is reminiscent of the (classical) Helmholtz decomposition of vector fields.

It is interesting to observe that not all $L^2(\mathbb R^{2n})$ functions have a fractional Helmholtz decomposition, which is notably different from the local case. Heuristically speaking, this is due to the fact that a function $w$ with a fractional Helmholtz decomposition should be written in the form
$$ w(x,y) = \alpha(x,y)(\varphi(x)-\varphi(y)) + \alpha(x,y)\times (A(y)-A(x)).$$
Because the right-hand side is a symmetric function of $x$ and $y$, it is clear that functions that are not symmetric can not have a fractional Helmholtz decomposition. By taking the scalar product with $\alpha$, we obtain that
\begin{equation}\label{eq:varphi}
    \varphi(x)-\varphi(y)=\frac{\alpha(x,y)\cdot w(x,y)}{|\alpha(x,y)|^2}.
\end{equation}  
If instead we take the vector product with $\alpha$, using the fact that $\mathcal G^*A=0$ we get
\begin{align*}
    \alpha(x,y)\times w(x,y) & =  \alpha(x,y)\times\bigg(\alpha(x,y)\times (A(y)-A(x))\bigg) 
    \\ & = 
    \alpha(x,y)(\alpha(x,y)\cdot (A(y)-A(x))) - (A(y)-A(x))|\alpha(x,y)|^2
    \\ & = 
    (A(x)-A(y))|\alpha(x,y)|^2,
\end{align*}
that is
\begin{equation}\label{eq:A} A(x)-A(y) = \frac{\alpha(x,y)\times w(x,y)}{|\alpha(x,y)|^2}. 
\end{equation}
Because the left-hand sides of formulas \eqref{eq:varphi} and \eqref{eq:A} have a special dependence on the variables $x,y$, it is clear that not all two-points vector fields $w$ can satisfy the equations above. In particular, not all two-points vector fields in $L^2(\mathbb R^{2n})$ have a fractional Helmholtz decomposition. \\

However, this way of reasoning is based on a pointwise formulation of the fractional gradient and curl, which does not always hold for Sobolev functions. For the sake of simplicity, we will give a rigorous discussion regarding the existence of fractional Helmholtz decompositions only in the case $s=1/2$.

Observe that because $\mathcal D^*, \mathcal C^*$ are defined by density for Sobolev functions starting from the pointwise formulas for smooth functions, we have that $\mathcal D^*\varphi$ and $\mathcal C^* A$ are respectively parallel and orthogonal to $\alpha$. Assume that $w\in H^r(\mathbb R^{2n})$ for $r$ large enough. We also assume that $w$ is parallel to $\alpha$, which means that it can be written as $\alpha f$ for a scalar function $f$. If $w$ has a fractional Helmholtz decomposition, then it must necessarily be $w=\mathcal D^*\varphi$. Let $\psi\in H^{1/2-r}$. Then
\begin{align*}
    \langle \mathcal Dw, \psi \rangle & 
    = \langle w, \mathcal D^*\psi \rangle =  \langle \mathcal D^*\varphi, \mathcal D^*\psi \rangle = \langle \varphi, \mathcal D\mathcal D^*\psi \rangle = \langle \varphi, (-\Delta)^{1/2}\psi\rangle = \langle (-\Delta)^{1/2}\varphi, \psi\rangle,
\end{align*}
which means that the equality $\mathcal Dw = (-\Delta)^{1/2}\varphi$ holds in weak sense. Here we used the mapping property for $\mathcal D$ obtained in \cite{Covi20b}. Thus $\varphi = I_{1}\mathcal D w$, which by the mapping properties of the Riesz potential implies that $\varphi \in H^{r+1/2}$. If $r$ is large enough, by the Sobolev  embedding theorem we deduce that $\varphi$ is continuous, and thus the pointwise formula 
$$ w(x,y)= \alpha(x,y)(\varphi(x)-\varphi(y)) $$
holds. By taking the scalar product with $\alpha$, we see that
\begin{equation*}
    \varphi(x)-\varphi(y)=\frac{\alpha(x,y)\cdot w(x,y)}{|\alpha(x,y)|^2} = f(x,y).
\end{equation*}
However, not all scalar functions of two variables $f(x,y)$ are of the form $\varphi(x)-\varphi(y)$ for a scalar function $\varphi$. This proves that not all functions in $L^2(\mathbb R^{2n})$ have a fractional Helmholtz decomposition. 

\subsection{Distributions with fractional Helmholtz decomposition}

We now generalize the above definition of Sobolev space with fractional Helmholtz decomposition to include distributions. In order to do so, we first need to define how to compute $\mathcal D^*$ and $\mathcal C^*$ of a distribution.

Let $T\in D'(\mathbb R^n)$ be a distribution. We define the adjoint of the fractional divergence of $T$ as the distribution $\mathcal D^* T$ such that
$$ \langle \mathcal D^* T, \varphi \rangle := \langle T, \mathcal D\varphi \rangle $$
holds for all $\varphi\in C^\infty_c(\mathbb R^{2n})$. Similarly, we define the adjoint of the fractional curl of a distribution $T\in D'(\mathbb R^n)$ as
$$ \langle \mathcal C^* T, \varphi \rangle := \langle T, \mathcal C\varphi \rangle, $$
for all $\varphi\in C^\infty_c(\mathbb R^{2n})$. Observe that the above definitions make sense in light of the mapping properties of the operators $\mathcal D, \mathcal C$. In fact, if $\varphi\in C^\infty_c(\mathbb R^{2n})$, then in particular $\varphi\in H^t(\mathbb R^{2n})$ for all $t\geq s$, and thus $\mathcal D\varphi, \mathcal C\varphi \in H^{t-s}(\mathbb R^n)$ for all $t\geq s$. For the fractional divergence this result is contained in lemma 2.7 in \cite{Covi20b}, while for the fractional curl it follows from our lemma \ref{second-lemma-curl}. Thus $\mathcal D\varphi, \mathcal C\varphi$ are smooth. By the pointwise definition of $\mathcal D, \mathcal C$ one immediately sees that $\mathcal D\varphi, \mathcal C\varphi$ are also compactly supported, and thus $\mathcal D\varphi, \mathcal C\varphi\in C^\infty_c(\mathbb R^n)$. This gives us the mapping properties
$$\mathcal D^* : D'(\mathbb R^n)\rightarrow D'(\mathbb R^{2n}),$$
$$\mathcal C^* : D'(\mathbb R^n)\rightarrow D'(\mathbb R^{2n}).$$
With this in mind, we can define the following space of distributions with fractional Helmholtz decomposition:
$$ D'_{fHd}(\mathbb R^{2n}) :=\{ \mathcal D^*T + \mathcal C^*S : T,S\in D'(\mathbb R^n), \nabla\cdot S =0 \}. $$
The condition $\nabla\cdot S=0$ of course means that $\langle S,\nabla\varphi\rangle=0$ for all $\varphi\in C^\infty_c(\mathbb R^{n})$.

\subsection{The projection operator $\Pi$}\label{subsec:pi}

Next, we define a projection operator $\Pi$ mapping two-points fields to one-point fields. Using $\Pi$, we will be able to compare the fractional and classical differential operators. \\

We let
$$(\Pi v)(x):=c\int_{\R^n}\frac{v(x,y)}{|x-y|^{n/2}}dy,$$ where $c$ is a constant to be defined later. We see that $\Pi v$ is well-defined in a pointwise sense for example for any two-points field $v\in C^\infty_c(\mathbb R^{2n})$, because local boundedness suffices in order to obtain a finite integral near the diagonal $\{y=x\}$. More generally, by the Hardy-Littlewood inequality it suffices that the slice function $y\mapsto v(x,y)$ belongs to $L^p$, $p\in(1,2)$, for all $x\in\mathbb R^n$. Problems can arise when $|y|\rightarrow\infty$, due to the fact that $|y|^{-n/2}$ is not integrable in the complement of a ball. However, we will be interested in computing $\Pi$ mostly for two-points fields which are fractional curls and gradients of Sobolev functions. To show that these are well-defined, we let $w\in C^\infty_c(\mathbb R^n)$ be a smooth, compactly supported one-point field and compute 
\begin{align*}
    \Pi(\mathcal D^*w)(x) & = \frac{c\, C_{n,s}^{1/2}}{\sqrt 2}\int_{\R^n}\frac{x-y}{|x-y|^{n+s+1}}(w(y)-w(x))dy
    \\ & =
    \frac{c\, C_{n,s}^{1/2}}{(1-n-s)\sqrt 2}\int_{\R^n}\nabla_y(|x-y|^{1-n-s})(w(y)-w(x))dy
    \\ & =
    \frac{c\, C_{n,s}^{1/2}}{(n+s-1)\sqrt 2}\int_{\R^n}\frac{\nabla w(y)}{|x-y|^{n+s-1}}dy
    \\ & =
    \frac{c\, C_{n,s}^{1/2}c_{1-s}}{(n+s-1)\sqrt 2}I_{1-s}\nabla w(x),
\end{align*}
where $I_{1-s}$ is the Riesz potential, and $c_{1-s}$ is the associated constant factor. We now choose $c>0$ in such way that $\Pi(\mathcal D^*w) = I_{1-s}\nabla w$ for all $w\in C^\infty_c(\mathbb R^n)$. Similarly, it holds that $\Pi(\mathcal C^*w) = I_{1-s}\nabla \times w$.

The Riesz potential $I_\sigma$, $\sigma\in(0,\frac n2)$, is known to map as $I_\sigma : H^{t}\rightarrow H^{t+\sigma}$ for all $t\in\mathbb R$. Thus, if $w\in C^\infty_c(\mathbb R^n)$ we have
$$ \|\Pi\mathcal D^*w\|_{H^{r-s}} = \|I_{1-s}\nabla w\|_{H^{r-s}} \leq \|\nabla w\|_{H^{r-1}} \leq \|w\|_{H^{r}}. $$
This allows us to extend the operator $\Pi\mathcal D^*$ by density, to operate as $\Pi\mathcal D^*: H^{r}\rightarrow H^{r-s}$ for all $r\in\mathbb R$. A similar result holds for $\Pi\mathcal C^*$. Thus the projection operator $\Pi$ is well-defined for all two-points fields which are fractional curls and gradients of Sobolev functions. Moreover, because for smooth, compactly supported two-points fields $\Pi$ is an integration in the variable $y$, it is clear that $\Pi(f(x)v(x,y))= f(x)\Pi v(x,y)$ holds for all $f\in C^\infty_c(\mathbb R^n)$. It is also immediately clear that $\Pi$ is well-defined on objects of the kind $f(x)\mathcal D^*w(x,y)$, where $f\in C^\infty_c(\mathbb R^n)$ and $w\in H^r$, $r\in\mathbb R$, as the above inequalities still hold. 

\begin{prop}\label{Pi mapping property}
    Let $s\in(0,1)$ and $r\in\mathbb R$. The projection operator $\Pi$ satisfies 
    $$ (\Pi \mathcal{D}^*) f =I_{1-s}(\nabla  f), \qquad (\Pi \mathcal{C}^*) f(x)=I_{1-s}(\nabla \times f) $$
    pointwise whenever $f\in C^\infty_c(\mathbb R^n)$, and in $H^{r-s}$ sense whenever $f\in H^r(\mathbb R^n)$. As a consequence, $\Pi\mathcal C^*\Pi\mathcal D^* \equiv 0$. The map $\Pi$ acts as a bounded linear operator between $H^r_{fHd}(\mathbb R^{2n})$ and $H^r(\mathbb R^n)$:
    \begin{align*}
  \Pi: &  \; H^r_{fHd}(\R^{2n} ) \to H^r(\R^{n}) \\
  &\; \mathcal{D}^*\varphi +\mathcal{C}^*A \mapsto I_{1-s}(\nabla \varphi + \nabla \times A)
\end{align*}

    Moreover, if $f\in H^r(\mathbb R^n)$ the following limits hold in $H^{r-1}$:

$$ \lim_{s\to 1^-} (\Pi \mathcal{D}^*) f(x) = \nabla f, \qquad \lim_{s\to 1^-} (\Pi \mathcal{C}^*) f(x) = \nabla \times f. $$

Finally, we have $(\Pi \mathcal{C}^*)^*=-\Pi \mathcal{C}^*$. 
\end{prop}

\begin{proof}
    We have already proved that $\Pi\mathcal D^* f = I_{1-s}\nabla f$ holds pointwise for $f\in C^\infty_c(\mathbb R^n)$. If now $f\in H^r(\mathbb R^n)$ for $r\in\mathbb R$, then $\Pi\mathcal D^* f \in H^{r-s}(\mathbb R^n)$. If $\psi\in H^{1-r}(\mathbb R^n)$, then by density and the properties of the Riesz potential, we have
    \begin{align*}
        \lim_{s\rightarrow 1^-}\langle \Pi\mathcal D^* f, \psi \rangle = \lim_{s\rightarrow 1^-}\langle I_{1-s}\nabla f, \psi \rangle  = \langle \nabla f,\psi \rangle .
    \end{align*}
    The same proof holds for $\Pi\mathcal C^*$. If now $f\in H^r(\mathbb R^n)$ and $\psi\in H^{s-r}(\mathbb R^n)$, we have
    \begin{align*}
        \langle \Pi\mathcal C^* f, \psi \rangle & = \langle I_{1-s}\nabla \times f, \psi \rangle = \langle \nabla \times f, I_{1-s}\psi \rangle
        = -\langle f, I_{1-s}\nabla \times \psi \rangle
        = -\langle f, \Pi\mathcal C^* \psi \rangle,
    \end{align*}
    which proves that $(\Pi\mathcal C^*)^* = -\Pi\mathcal C^*$.
\end{proof}

Next, we show that the projection $\Pi$ gives a bijection between one-point functions, and two-points functions that have a fractional Helmholtz decomposition. 

\begin{prop}{\label{Bijectivity of Pi}}
  Let $s\in(0,1)$ and $r\in\mathbb R$. The projection operator $\Pi$  
\begin{align*}
  \Pi: &  \; H^r_{fHd}(\R^{2n} ) \to H^r(\R^{n}) \\
  &\; \mathcal{D}^*\varphi +\mathcal{C}^*A \mapsto I_{1-s}(\nabla \varphi + \nabla \times A)
\end{align*}
is bijective.
\end{prop}

\begin{proof}
If $\tilde{v}\in H^r(\R^n)$, then there exist a scalar and a vector potentials $\tilde\varphi, \tilde A \in H^{r+1}(\R^n)$ such that  $\tilde{v}=\nabla \tilde \varphi +\nabla \times \tilde A$, with $\nabla\cdot \tilde A=0$. It suffices to find $\varphi, A \in H^{r+s}$ such that
$$ \Pi \mathcal D^* \varphi = \nabla\tilde\varphi, \qquad  \Pi \mathcal C^* A = \nabla\times\tilde A, \qquad \nabla\cdot A=0, $$
because then $v:= \mathcal D^*\varphi + \mathcal C^*A$ belongs to $H^r_{fHd}(\R^{2n} )$, and moreover $\Pi v = \tilde v$.
By defining
$$ \varphi := (-\Delta)^{\frac{1-s}2}\tilde\varphi, \qquad A := (-\Delta)^{\frac{1-s}2}\tilde A, $$
we see that $\varphi,A$ belong to $H^{r+s}$ by the mapping properties of the fractional Laplacian. Moreover, by Lemma \ref{Pi mapping property} we have
$$ \Pi\mathcal D^*\varphi = I_{1-s}\nabla\varphi = I_{1-s}\nabla (-\Delta)^{\frac{1-s}2}\tilde\varphi = \nabla\tilde\varphi, $$
and similarly for the curl. Finally, we see that $$\nabla\cdot A = (-\Delta)^{\frac{1-s}2}\nabla\cdot\tilde A =0.$$ This proves that $\Pi:  H^r_{fHd}(\R^{2n} ) \to H^r(\R^{n})$ is surjective.
Assume now that $v\in H^r_{fHd}(\R^{2n} )$ verifies $\Pi v =0$. Then 
$$ 0=\Pi v= \Pi\mathcal D^*\varphi + \Pi\mathcal C^*A = I_{1-s}(\nabla\varphi + \nabla\times A), $$
and by taking the fractional Laplacian on both sides we are left with
$$ \nabla\varphi + \nabla\times A =0. $$
Because of the assumption $\nabla\cdot A=0$, this (classical) Helmholtz decomposition is unique, and thus $\varphi=A=0$. This implies that $v=0$, which means that $\Pi:  H^r_{fHd}(\R^{2n} ) \to H^r(\R^{n})$ is injective.
\end{proof}

By the same arguments as above, we can prove the uniqueness of the fractional Helmholtz decomposition:

\begin{lemma}\label{unique representation}
    Let $v\in H^r_{fHd}(\mathbb R^{2n})$. There exists a unique pair $\varphi, A$ of potentials such that $\varphi, A\in H^{r+s}$, $\nabla\cdot A\equiv 0$, and $v=\mathcal D^*\varphi + \mathcal C^*A$.
\end{lemma}  

\begin{proof}
    Let $\varphi', A'$ be a second pair of potentials satisfying the conditions in the statement. Then 
    $$ 0=\Pi\mathcal D^*(\varphi-\varphi') + \Pi\mathcal C^*(A-A') = I_{1-s}(\nabla(\varphi-\varphi') + \nabla\times(A-A')), $$
    from which we deduce $\varphi=\varphi'$ and $A=A'$ by using the assumption that $\nabla\cdot(A-A')=0$.
\end{proof}

\color{black}

\section{The fractional Maxwell equations}\label{Section: fractional Maxwell equations}

We define the two-point to two-point fractional curl $\widetilde{\mathcal C}:= \Pi^{-1}(\Pi\mathcal C^*)\Pi : H^r_{fHd}(\R^{2n}) \to H^{r-s}_{fHd}(\R^{2n} )
$, and prove the following lemma:

\begin{lemma}\label{lemma:tilde-c}
    The two-point to two-point fractional curl $\widetilde{\mathcal C}$ maps between functions with fractional Helmholtz decomposition as
    $$\widetilde{\mathcal C} := \Pi^{-1}(\Pi\mathcal C^*)\Pi : H^r_{fHd}(\R^{2n}) \to H^{r-s}_{fHd}(\R^{2n} ). 
$$
Moreover, given $H^r_{fHd}(\R^{2n})\ni v = \mathcal D^*\varphi + \mathcal C^*A$, we have $$\mathcal{\widetilde C}v = \mathcal C^* \Pi \mathcal C^* A,
$$
that is $\widetilde{\mathcal C}$ behaves as $\mathcal C^*\Pi$ on the curl part of $v$, and it does not depend on the gradient part of $v$. In particular, for all $v\in H^r_{fHd}(\R^{2n})$ it holds that $$\mathcal D \widetilde {\mathcal C} v =0.$$
\end{lemma}

\begin{proof}
The stated mapping property makes sense, as by Proposition \ref{Pi mapping property} we have
$$
v\in H^r_{fHd}(\R^{2n}) \implies \Pi v\in H^r(\R^n) \implies (\Pi \mathcal{C}^*)\Pi v \in H^{r-s}(\R^n) \implies \widetilde{\mathcal C}v \in H^{r-s}_{fHd}(\R^{2n})
$$
for all $r \in \R$. Moreover, every $v\in H^r_{fHd}(\R^{2n})$ can be written in the form $v =\mathcal D^*\varphi + \mathcal C^*A$, with $\nabla\cdot A=0$. Then by Proposition \ref{Pi mapping property}, we have
$$ (\Pi \mathcal C^*)(\Pi v) = \Pi \mathcal C^* \Pi \mathcal D^* \phi + \Pi \mathcal C^* \Pi \mathcal C^* A  = \Pi \mathcal C^* \Pi \mathcal C^* A.  $$
Hence 
$\widetilde {\mathcal C} v = \Pi^{-1}(\Pi \mathcal C^*)(\Pi v) = (\Pi^{-1}\Pi) (\mathcal C^* \Pi \mathcal C^* A)$.
Now since 
$$ \nabla\cdot \Pi \mathcal C^* A = \nabla\cdot I_{1-s}\nabla\times A = I_{1-s}\nabla\cdot\nabla\times A =0,$$
we have $\mathcal C^*\Pi \mathcal C^* A\in H^{r-s}_{fHd}(\mathbb R^{2n})$. Therefore, since $\Pi$ is bijective on the space of functions with fractional Helmholtz decomposition, we obtain 
$$
\widetilde {\mathcal{C}} v = \Pi^{-1}(\Pi \mathcal C^* \Pi \mathcal C^* A) =  \mathcal C^* \Pi \mathcal C^* A.
$$
In particular, this implies that
$\mathcal D \widetilde {\mathcal C} v =0$,
since $\mathcal D \mathcal C^*=0$. 
\end{proof}

\begin{rmk}
    Observe that we also have $$\widetilde{\mathcal C}\widetilde{\mathcal C} =  \Pi^{-1}(\Pi \mathcal{C}^*)(\Pi \mathcal{C}^*)\Pi: H^r_{fHd}(\R^{2n})\to H^{r-2s}_{fHd}(\R^{2n}). $$
\end{rmk}

With this in mind, let us now return to the fractional Maxwell system
\begin{equation}
    \begin{cases}
        \frac{\partial \mathbf (\varepsilon \mathbf E)}{\partial t} - \widetilde{\mathcal{C}}\mathbf H = -\mathbf J, \\
        \frac{\partial( \mu \mathbf H)}{\partial t} + \widetilde{\mathcal C}\mathbf E = 0, \\
        \mathcal D (\mu \mathbf H)=0, \\
        \mathcal D (\varepsilon \mathbf E)=\rho,
    \end{cases}
\end{equation}
which we want to study in the space of functions with fractional Helmholtz decomposition. By the first and last equations, using the fact that $\mathcal D \widetilde{\mathcal C}=0$ we obtain the law of  conservation of charge
$$\partial_t\rho + \mathcal{D}\mathbf{J} =0.
$$
This is consistent with \cite{Covi20}: if, as expected by Ohm's law, the current $\mathbf J$ takes the form $\Theta\nabla^s\varphi$ for a scalar potential $\varphi$ and a conductivity $\Theta$, in stationary conditions one recovers the fractional conductivity equation $$ \mathcal D (\Theta\nabla^s\varphi)=0. $$
Assume now that the electric field, the magnetic field, the current and the charge density take the form
\begin{align*}
\mathbf{E}(x,y,t)&=e^{-i\omega t}\,\hat E(x,y),\\
\mathbf{H}(x,y,t)&=e^{-i\omega t}\,\hat H(x,y),\\
\mathbf{J}(x,y,t)&=e^{-i\omega t}\,\hat J(x,y),\\
\rho(x, t)&=e^{-i\omega t}\,\hat\rho(x).
\end{align*}
The fractional Maxwell equations then become
\begin{equation*}
    \begin{cases}
    -\,i\omega\,\varepsilon \hat E-\widetilde{\mathcal{C}}\hat H=-\hat J,\\
    -\,i\omega\,\mu \hat H+\widetilde{\mathcal{C}}\hat E=0,\\
    \mathcal{D}(\mu\hat H)=0,\\
    \mathcal{D}(\varepsilon\hat E)=\hat\rho,
\end{cases}
\end{equation*}
and the law of conservation of charge becomes $i\omega\hat\rho = \mathcal{D}\mathbf{\hat J}$.
Define
$$\begin{aligned}
    E := \varepsilon_0^{1/2}\hat E,
\qquad
H := \mu_0^{1/2}\hat H,
\qquad
F := ik\mu_0^{1/2}\hat J,
\end{aligned}$$
$$\begin{aligned}
k := (\mu_0\varepsilon_0)^{1/2}\omega,
\qquad
\varepsilon_r :=\varepsilon_0^{-1}\varepsilon,
\qquad
\mu_r := \mu_0^{-1}\mu,
\end{aligned}$$
where $\varepsilon_0, \mu_0$ respectively represent the constant background electric permittivity and magnetic permeability.
Then using the law of conservation of charge we can rewrite the fractional Maxwell equations as
\begin{equation}\label{eq_timehm_fractionalmaxwell}
    \begin{cases}
        -ik \varepsilon_rE - \widetilde{\mathcal{C}} H = -(ik)^{-1}F, \\
        -ik \mu_r H+ \widetilde{\mathcal C} E = 0, \\
        \mathcal D (\mu_r H)=0, \\
        \mathcal D (\varepsilon_r E)=(ik)^{-2}\mathcal{D}F.
    \end{cases}
\end{equation}
Observe that the third and fourth equations are made redundant by the second and first one respectively, given that $\mathcal D\widetilde{\mathcal C}=0$. Thus the system can be rewritten as

\begin{equation}
    \begin{cases}
        ik \varepsilon_rE + \widetilde{\mathcal{C}} H = (ik)^{-1}F, \\
        \mu_r H = (ik)^{-1}\widetilde{\mathcal C}E,
    \end{cases}
\end{equation}
which is equivalent to 

\begin{equation}\label{12}
    \begin{cases}
        \widetilde{\mathcal C} (\mu_r^{-1}\widetilde{\mathcal C} E)-k^2\varepsilon_rE=F, \\
         H = (ik\mu_r)^{-1}\widetilde{\mathcal C}E.
    \end{cases}
\end{equation}

We shall always assume that $\mu_r\equiv 1$, which is reasonable, as the magnetic permeability usually changes much more slowly than the electric permittivity. This assumption is classically used for the local Maxwell equations as well. This gives us the system

\begin{equation}\label{final-after-computations}
    \begin{cases}
        \widetilde{\mathcal C} \widetilde{\mathcal C} E-k^2\varepsilon_rE=F, \\
         H = (ik)^{-1}\widetilde{\mathcal C}E,
    \end{cases}
\end{equation}
which consists of an equation for the electric field $E$, and a second equation that allows us to compute the magnetic field $H$ once $E$ is known. Observe that both equations hold in $\mathbb R^{2n}$.

\subsection{The fractional Maxwell equations in a vacuum}

In a vacuum, because the current density $\mathbf{J}$ vanishes, we have $F=0$. Since in this case there is also no matter, we have $\varepsilon_r=1$. Thus in the case of a vacuum the system simplifies to
\begin{equation}\label{34}
    \begin{cases}
        \widetilde{\mathcal C}  \widetilde{\mathcal C}  E-k^2E=0, \\
         H = (ik)^{-1}\widetilde{\mathcal C}E.
    \end{cases}
\end{equation}

We assume that the incident field $E_i$ of our scattering problem solves the equation
$$ \widetilde{\mathcal C}  \widetilde{\mathcal C}  E_i-k^2E_i=0. $$
We will show that this equation is solved by $E_i := \mathcal C^*\nabla\times (pe^{ik^{1/s}d\cdot x})$, where the vectors $p,d\in\mathbb R^n$ respectively are the polarization and direction of the \emph{nonlocal plane wave} $pe^{ik^{1/s}d\cdot x}$. This is in complete agreement with the local case, in which the incident wave $E_i$ typically takes the form $\nabla\times\nabla\times(pe^{ik d\cdot x})$, as discussed in \cite{CCMS10}. However, because the plane wave does not belong to any Sobolev space, we can not use the in this case the definition of $\widetilde{\mathcal C}$ we have given above. 

In general, our incident wave $E_i$ is a distribution belonging to $\mathcal C^*(S'(\mathbb R^n)) \subseteq D'(\mathbb R^{2n})$. In particular, it belongs to $D'_{fHd}(\mathbb R^{2n})$, and it is of fractional curl form. In order to define $\widetilde{\mathcal C}E_i$, following Lemma \ref{lemma:tilde-c} we will let 
$$ \widetilde{\mathcal C} : \mathcal C^*(S'(\mathbb R^n)) \ni \mathcal C^* T \mapsto \mathcal C^* (\Pi \mathcal C^*) T, $$
where $(\Pi\mathcal C^*) T$ is defined in the sense of distributions as
$$ \langle (\Pi\mathcal C^*) T ,\varphi\rangle := -\langle T , I_{1-s}\nabla\times \varphi\rangle, $$
for all $\varphi\in C^\infty_c(\mathbb R^n)$. Observe that this last definition makes sense, because $T$ is a tempered distribution, and the Riesz potential of a smooth, compactly supported function is itself smooth and has at most polynomial growth. With this definition,
\begin{align*}
    \widetilde{\mathcal C} E_i & = \widetilde{\mathcal C} \mathcal C^*\nabla\times (pe^{ik^{1/s}d\cdot x}) = \mathcal C^* (\Pi \mathcal C^*) \nabla\times (pe^{ik^{1/s}d\cdot x}),
\end{align*}
and
\begin{align*}
    \widetilde{\mathcal C}\widetilde{\mathcal C} E_i & =  \widetilde{\mathcal C}\mathcal C^* (\Pi \mathcal C^*) \nabla\times (pe^{ik^{1/s}d\cdot x})= \mathcal C^* (\Pi \mathcal C^*)(\Pi \mathcal C^*) \nabla\times (pe^{ik^{1/s}d\cdot x}).
\end{align*}
If $\varphi\in C^\infty_c(\mathbb R^n)$, then
\begin{align*}
    \langle(\Pi \mathcal C^*)(\Pi \mathcal C^*) \nabla\times (pe^{ik^{1/s}d\cdot x}),\varphi\rangle & = \langle \nabla\times (pe^{ik^{1/s}d\cdot x}),I_{2-2s}\nabla\times\nabla\times\varphi\rangle 
    \\ & = 
    -\langle  pe^{ik^{1/s}d\cdot x},I_{2-2s}\nabla\times\nabla\times\nabla\times\varphi\rangle
    \\ & = 
    -\langle  pe^{ik^{1/s}d\cdot x},I_{2-2s}\nabla\times( \nabla\nabla\cdot - \Delta )\varphi\rangle
    \\ & = 
    -\langle  pe^{ik^{1/s}d\cdot x},\nabla\times( - \Delta )^s\varphi\rangle
    \\ & = 
    \langle  \nabla\times( - \Delta )^s(pe^{ik^{1/s}d\cdot x}),\varphi\rangle.
\end{align*}
Thus the fractional Maxwell equations in a vacuum becomes
$$ 0 = \widetilde{\mathcal C}  \widetilde{\mathcal C}  E_i-k^2E_i = \mathcal C^* \nabla\times\big( ( - \Delta )^s - k^2 \big)  (pe^{ik^{1/s}d\cdot x}), $$
which holds because 
$$ \big( ( - \Delta )^s - k^2 \big)  (pe^{ik^{1/s}d\cdot x})=0 $$
for all $p,d\in\mathbb R^n$. Therefore, we see that $E_i = \mathcal C^*\nabla\times (pe^{ik^{1/s}d\cdot x})$ is a good choice of incoming field. We also prove the following regularity lemma about the incident field:

\begin{lemma}\label{Pi_Ei}
    If $\sigma\in C^\infty_c(\mathbb R^n)$, then $\sigma\Pi E_i \in L^2(\mathbb R^n)$ for all $p,d\in \mathbb S^n\setminus\{0\}$.
\end{lemma}

\begin{proof}
    Because smooth compactly supported functions are dense in $L^2$, it suffices to show that 
    $$ |\langle \sigma \Pi E_i, \varphi \rangle| \lesssim \|\varphi\|_{L^2} $$
    holds for all $\varphi\in C^\infty_c(\mathbb R^n)$. By the definitions of $E_i$ and of the operator $\Pi \mathcal C^*$ on tempered distributions, we have
    \begin{align*}
        |\langle \sigma \Pi E_i, \varphi \rangle| & = |\langle (\Pi \mathcal C^*)\nabla\times (pe^{ik^{1/s}d\cdot x}), \sigma\varphi \rangle|
        \\ & =
        |\langle \nabla\times (pe^{ik^{1/s}d\cdot x}), I_{1-s}\nabla\times(\sigma\varphi) \rangle|
        \\ & =
        |\langle \nabla\times\nabla\times (pe^{ik^{1/s}d\cdot x}), I_{1-s}(\sigma\varphi) \rangle|
        \\ & =
        k^{2/s}|p-(p\cdot d)d|\, |\langle e^{ik^{1/s}d\cdot x} , I_{1-s}(\sigma\varphi) \rangle|
        \\ & =
        k^{2/s}|p-(p\cdot d)d|\, |\langle I_{1-s}(e^{ik^{1/s}d\cdot x}) , \sigma\varphi \rangle|
        \\ & =
        k^{2/s}|k^{1/s}d|^{s-1}|p-(p\cdot d)d|\, |\langle e^{ik^{1/s}d\cdot x}\sigma,\varphi \rangle|
        \\ & \leq 
        k^{1+1/s}|p-(p\cdot d)d|\, \int_{\mathbb R^n} |\sigma| |\varphi| dx
        \\ & \leq
        k^{1+1/s} \|\sigma\|_{L^2}\|\varphi\|_{L^2}.
    \end{align*}
The significance of $\sigma$ is that the implied constant depends only on $k$ and $\sigma$, but not on $\varphi$.
\end{proof}

\subsection{The scattering problem} We consider the scattering problem for equation \eqref{final-after-computations} in the absence of currents. Let
\[
E = E_i + E_s,
\]
where
 $E_i$ is the incident field computed in the previous section, and $E_s$ is the \emph{scattered field}, generated by the inhomogeneity in
  $\varepsilon_r = \varepsilon_r(x)$. We assume that {$\varepsilon_r$ satisfies assumption (H) as in the introduction}. The scattered field solves the equation
  \begin{equation}\label{eq:E_s}
\widetilde {\mathcal C} \widetilde {\mathcal C} E_s - k^2 \varepsilon_r E_s
=
k^2(\varepsilon_r - 1)E_i.
\end{equation}
We will look for solutions to equation \eqref{eq:E_s} in the space $ H_{fHd}^{s}(\R^{2n})$, where the operator $\widetilde {\mathcal C} := \Pi^{-1}(\Pi\mathcal C^*)\Pi$ is well-defined and maps as
$$ \widetilde {\mathcal C} : H_{fHd}^{s}(\R^{2n}) \rightarrow L_{fHd}^{2}(\R^{2n}). $$
Recall that we also have $\widetilde {\mathcal C}\widetilde {\mathcal C} = \Pi^{-1}(\Pi\mathcal C^*)^2\Pi$, and
$$ \widetilde {\mathcal C}\widetilde {\mathcal C} : H_{fHd}^{s}(\R^{2n}) \rightarrow H_{fHd}^{-s}(\R^{2n}). $$
The structure of the operator $\widetilde {\mathcal C}\widetilde {\mathcal C}$ suggests that we should take $\Pi$ of equation \eqref{eq:E_s}. As observed in Section \ref{subsec:pi}, this operation is well-defined on the left-hand side of the equation. On the right-hand side, we observe that $\Pi E_i = (\Pi \mathcal C^*)\nabla\times (pe^{ik^{1/s}x\cdot d})$, which is well-defined as an element of $D'(\mathbb R^n)$. By letting $$\widetilde E_s := \Pi E_s, \qquad \widetilde E_i := \Pi E_i,$$
we obtain
\begin{equation}\label{eq:one-point}
    (\Pi\mathcal C^*)^2\widetilde E_s - k^2 \varepsilon_r \widetilde E_s
=
k^2(\varepsilon_r - 1)\widetilde E_i,
\end{equation}
which holds in $\mathbb R^n$. By the mapping properties of $\Pi$, we see that $\widetilde E_s \in H^s(\mathbb R^n)$. Then $(\Pi\mathcal C^*)^2\widetilde E_s\in H^{-s}(\mathbb R^n)$, and the equation can be rewritten as
\begin{equation}\label{eq:1-point}
    I_{2-2s}\nabla\times\nabla\times\widetilde E_s - k^2 \varepsilon_r \widetilde E_s
=
k^2(\varepsilon_r - 1)\widetilde E_i.
\end{equation}  

\begin{lemma}\label{equivalence-lemma}
    Assume {$\varepsilon_r$ satisfies assumption (H)}, and that $0$ is not a Dirichlet eigenvalue for the problem $$(-\Delta)^su = k^2\varepsilon_ru, \qquad \mbox{in } \R^n.$$ Then equation \eqref{eq:1-point} can be equivalently rewritten as
    \begin{align}\label{EsEi}
    (-\Delta)^s\widetilde E_s +P_{2s-1}\widetilde E_s-k^2\varepsilon_r \widetilde{E}_s=k^2(\varepsilon_r-1)\widetilde{E}_i- P_{2s-1} \widetilde E_i \qquad \mbox{in } \R^n,
\end{align}
where $P_{2s-1}$ is the operator of order $2s-1$ given by 
\begin{align*}
P_{2s-1} u :& =
- (-\Delta)^{s-1}\nabla\left(  \nabla (\log \varepsilon_r)\cdot u \right).
\end{align*}
\end{lemma}

\begin{proof}
By taking the divergence of equation \eqref{eq:1-point}, we see that
\begin{align*}
    0 & = \nabla\cdot\bigg((\varepsilon_r - 1)\widetilde E_i + \varepsilon_r \widetilde E_s\bigg)
    \\ & =
    \nabla(\varepsilon_r - 1)\widetilde E_i + (\varepsilon_r - 1)\nabla\cdot\widetilde E_i  +  \nabla\varepsilon_r \widetilde E_s  +  \varepsilon_r \nabla\cdot\widetilde E_s
    \\ & =
    \nabla\varepsilon_r(\widetilde E_i  +  \widetilde E_s)  +  \varepsilon_r \nabla\cdot\widetilde E_s,
\end{align*}
where we used the fact that $\nabla\cdot\widetilde E_i=0$. Hence
\begin{align*}
      \nabla\cdot\widetilde E_s =  -  \frac{\nabla\varepsilon_r}{\varepsilon_r}\cdot(\widetilde E_i  +  \widetilde E_s) = -\nabla(\log \varepsilon_r)\cdot(\widetilde E_s + \widetilde E_i),
\end{align*}
and
$$ \nabla\nabla\cdot \widetilde E_s = -\nabla\big(\nabla(\log \varepsilon_r)\cdot\widetilde E_s\big) -\nabla\big(\nabla(\log \varepsilon_r)\cdot\widetilde E_i\big). $$
This allows us to simplify the operator in equation \eqref{eq:1-point} in the following way:
\begin{align*}
    I_{2-2s}\nabla\times\nabla\times \widetilde E_s & = (-\Delta)^{s-1}(-\Delta+\nabla\nabla\cdot)\widetilde E_s
    \\ & =
    (-\Delta)^s\widetilde E_s + (-\Delta)^{s-1}\nabla\nabla\cdot\widetilde E_s
    \\ & =
    (-\Delta)^s\widetilde E_s +P_{2s-1}\widetilde E_s +P_{2s-1}\widetilde E_i,
\end{align*}
which gives equation \eqref{EsEi}. Conversely, assume that $\widetilde E_s\in H^s(\mathbb R^n)$ solves equation \eqref{EsEi}, and define $A:= k^2\varepsilon_r \widetilde E_s + k^2(\varepsilon_r - 1)\widetilde E_i$. We see by Lemma \ref{Pi_Ei} that $A\in L^2(\mathbb R^n)$, and thus $\nabla\cdot A\in H^{-1}(\mathbb R^n)$, where
$$\nabla\cdot A = k^2\nabla\cdot (\varepsilon_r \widetilde E_s + (\varepsilon_r - 1)\widetilde E_i) =  k^2 (\varepsilon_r \nabla\cdot\widetilde E_s  + \nabla\varepsilon_r \cdot(\widetilde E_s+\widetilde E_i)).$$
Equation \eqref{EsEi} can then be written as follows:
\begin{align*}
    0 & = (-\Delta)^s\widetilde E_s +P_{2s-1}(\widetilde E_s+\widetilde E_i)-A
    \\ & =
    I_{2-2s}(-\Delta)\widetilde E_s - (-\Delta)^{s-1}\nabla\left(  \nabla (\log \varepsilon_r)\cdot (\widetilde E_s+\widetilde E_i) \right)-A
    \\ & =
    I_{2-2s}\nabla\times\nabla\times\widetilde E_s - A -(-\Delta)^{s-1}\nabla\left(\nabla\cdot\widetilde E_s +   \frac{\nabla\varepsilon_r}{\varepsilon_r}\cdot (\widetilde E_s+\widetilde E_i)\right)
    \\ & =
    I_{2-2s}\nabla\times\nabla\times\widetilde E_s - A -(-\Delta)^{s-1}\nabla\left(\frac{\nabla\cdot A}{k^2 \varepsilon_r}\right).
\end{align*}
Thus it suffices to show that $\nabla\cdot A =0$. By taking the divergence of the last equation and defining $v:= \frac{\nabla\cdot A}{k^2 \varepsilon_r}$, we see that
$$  (-\Delta)^{s}v = k^2 \varepsilon_r v,  \qquad \mbox{in } \R^n. $$
Because $v\in H^{-1}(\mathbb R^n)$, by the mapping property of the Riesz potential and the previous equation, we see that it must be $v\in H^{2s-1}(\mathbb R^n)$. By a bootstrap argument, we deduce that $v\in H^t$ for all $t>0$, and in particular $v\in H^s$. Since $0$ is not a Dirichlet eigenvalue for the above problem, we finally deduce $v\equiv 0$, which implies $\nabla\cdot A\equiv 0$.
\end{proof}

\begin{rmk}
    Writing equation \eqref{eq:1-point} in the form given by \eqref{EsEi} has the positive effect of showing that the main part of the operator acting on $\widetilde E_s$ is the fractional Laplacian, which does not depend on the coefficient $\varepsilon_r$.
\end{rmk}

In the next section we will prove the well-posedness of the above equation \eqref{EsEi} in the weighted Sobolev space $H^s_\delta(\mathbb R^n)$, using an argument based on the Lax-Milgram theorem. The weight $\delta>0$ is made necessary by the fact that the Rellich lemma does not hold on unbounded sets in unweighted Sobolev spaces. For the two-points equation \eqref{eq:E_s} we assume the following weak concept of solution in the space of Sobolev functions with fractional Helmholtz decomposition:

\begin{defin}\label{def:well-posedness-2-points}
    Let $p,d\in\mathbb S^n\setminus\{0\}$, and define the incident field $E_i := \mathcal C^*\nabla\times(pe^{ik^{1/s}x\cdot d}) \in \mathcal C^*(S'(\mathbb R^n)) \subseteq D'(\mathbb R^{2n})$. A scattered field $E_s \in H^s_{fHd}(\mathbb R^{2n})$ is a solution of the two-points equation
    $$\widetilde {\mathcal C} \widetilde {\mathcal C} E_s - k^2 \varepsilon_r E_s = k^2(\varepsilon_r - 1)E_i \qquad \mbox{in } \mathbb R^{2n}$$
    if and only if $\widetilde E_s := \Pi E_s \in H^s_\delta$ solves the one-point equation
    $$(-\Delta)^s\widetilde E_s +P_{2s-1}\widetilde E_s-k^2\varepsilon_r \widetilde{E}_s=k^2(\varepsilon_r-1)\widetilde{E}_i- P_{2s-1} \widetilde E_i \qquad \mbox{in } \R^n.$$
\end{defin}

Using this concept of solution, we will show that the uniqueness in the direct problem for equation \eqref{eq:E_s} can be deduced from the uniqueness in the direct problem for equation \eqref{EsEi}.

\section{Proofs of the Rellich lemma and well-posedness}\label{section: well-posedness}
\subsection{Equation for one-point fields and well-posedness}

In this section we will prove the uniqueness result for the one-point equation \eqref{EsEi}. We start by showing a Rellich-type lemma for weighted fractional Sobolev spaces:

\begin{prop}[Rellich lemma for weighted fractional Sobolev spaces]\label{Rellich_1}
    Let $s\in(0,1)$ and $\delta >0$. The embedding $H^s_\delta \hookrightarrow L^2$ is compact.
\end{prop}

\begin{proof}
    We want to show that every bounded sequence in $H^s_\delta$ admits a strongly converging subsequence in $L^2$. To this end, let $\{u_j\}_{j\in\mathbb N}\subset H^s_\delta$ be a sequence for which there exists $C>0$ such that $\|u_j\|_{H^s_\delta}\leq C$ holds for all $j\in\mathbb N$. For all $N\in\mathbb N$ we see that
    \begin{align*}
         \|u_j\|_{H^s(B_N)} \leq \|u_j\|_{H^s} \leq \|u_j\|_{H^s_\delta},   \end{align*}
    and thus the sequence $\{u_j\}_{j\in\mathbb N}$ is bounded also in $H^s(B_N)$ by the same constant $C$. By the Rellich lemma in unweighted fractional Sobolev spaces, the embedding $H^s(B_N)\hookrightarrow L^2(B_N)$ is compact, and thus there exists a subsequence, indicated by $\{v_j\}_{j\in\mathbb N}$, which converges strongly in $L^2(B_N)$. In particular, $\{v_j\}_{j\in\mathbb N}$ is a Cauchy sequence in $L^2(B_N)$, and thus we can find $j,k\in\mathbb N$ so large that $\|v_j-v_k\|_{L^2(B_N)}\leq N^{-\delta}$. On the other hand,
    \begin{align*}
        \|v_j-v_k\|_{L^2(\mathbb R^n\setminus B_N)}^2 &= \int_{\mathbb R^n\setminus B_N} \frac{|v_j(x)-v_k(x)|^2\langle x\rangle^{2\delta}}{(1+|x|^2)^\delta} dx 
        \\ & \leq  
        N^{-2\delta} \int_{\mathbb R^n} |v_j(x)-v_k(x)|^2\langle x\rangle^{2\delta} dx 
        \\ & \leq 
        N^{-2\delta}\|v_j-v_k\|_{H^s_\delta}^2 \leq 4C^2N^{-2\delta}.
    \end{align*}
    Therefore, for $j,k$ large enough it holds that
    \begin{align*}
        \|v_j-v_k\|^2_{L^2} = \|v_j-v_k\|^2_{L^2(B_N)} + \|v_j-v_k\|^2_{L^2(\mathbb R^n\setminus B_N)}  \leq (1+4C^2)N^{-2\delta}.
    \end{align*}
    With these considerations in mind, let $\{u^{(0)}_j\}_{j\in\mathbb N} := \{u_j\}_{j\in\mathbb N}$ be the original sequence, and for all $N\in\mathbb N$ let $\{u^{(N+1)}_j\}_{j\in\mathbb N}$ be extracted from $\{u^{(N)}_j\}_{j\in\mathbb N}$ according to the above procedure. Finally, let $\{w_j\}_{j\in\mathbb N}$ be extracted from $\{u_j\}_{j\in\mathbb N}$ by the diagonal argument, that is the sequence such that
    $$ w_j = u^{(j)}_j, \qquad \mbox{for all } j\in\mathbb N. $$
    Observe that the sequence $\{w_j\}_{j\in\mathbb N}$ is eventually contained in all of the sequences $\{u^{(N)}_j\}_{j\in\mathbb N}$, $N\in\mathbb N$, and thus for all $\varepsilon >0$ it is possible to find $j,k$ so large that
    $$ \|w_j-w_k\|_{L^2} \leq \varepsilon. $$
    In fact, it suffices to choose $j,k$ as for the sequence $\{u^{(N_\varepsilon)}_j\}_{j\in\mathbb N}$, where $N_\varepsilon :=\lfloor(1+4C^2)^{1/2\delta}\varepsilon^{-1/2\delta}\rfloor+1$. Thus we see that the sequence $\{w_j\}_{j\in\mathbb N}$ is Cauchy in $L^2$. Because $L^2$ is complete, the sequence converges strongly in $L^2$, which proves that the embedding $H^s_\delta \hookrightarrow L^2$ is compact.
\end{proof}
One can generalize this proof slightly to obtain the following result: 
\begin{prop}[Generalized Rellich lemma for weighted fractional Sobolev spaces]\label{generalize_rellich}
    Let $s\in(0,1)$ and $\delta_1 > \delta_2 \geq 0$. The embedding $H^s_{\delta_1} \hookrightarrow L^2_{\delta_2}$ is compact.
\end{prop}

Using the Rellich lemma for weighted Sobolev spaces, we now prove the well-posedness (in the weak sense) of the equation
\begin{equation}\label{eq_1}
    (-\Delta)^s u + P_{2s-1}u - k^2\varepsilon_r u = F.
\end{equation} 
We start with the following lemma:

\begin{lemma}\label{bilinearwell-defined}
    Let $\delta>0, s\in [\frac{1}{2},1)$, {and $\varepsilon_r$ satisfies assumption (H)}. The bilinear form $B: H^s_\delta \times H^s_\delta \rightarrow \mathbb R$ given by
$$B(u,v) := \langle (-\Delta)^{s} u, v\rangle_{L^2_\delta} + \langle P_{2s-1}u,v\rangle_{L^2_\delta} - k^2\langle \varepsilon_r u, v\rangle_{L^2_\delta}$$
is well-defined.
\end{lemma}

\begin{proof}
    We need to observe that $|B(u,v)|$ is finite for all $u,v\in H^s_\delta$. Let $w$ denote the weight $w:=(1+|x|^2)^{\delta}$. The first term in the bilinear form gives 
\begin{align*}
\left\langle(-\Delta)^{s} u, v\right\rangle_{L_{\delta}^{2}} & =\int_{\mathbb{R}^{n}} w\left((-\Delta)^{s} u\right) v d x\\
&=\int_{\mathbb{R}^{n}}\left((-\Delta)^{s / 2} u\right) \left((-\Delta)^{s / 2}(v w) \right) d x \\
& =\int_{\mathbb{R}^{n}}(-\Delta)^{s / 2} u\left(w(-\Delta)^{s / 2} v+R_{s-1}^{w} v\right) d x,
\end{align*}
where the operator $R_{s-1}^{w}:=[(-\Delta)^{s/2},w]$ has order $s-1$ and principal symbol $-i \nabla_{\xi}\left(|\xi|^{s}\right) \cdot \nabla_{x} w=-i s|\xi|^{s-2} \xi \cdot \nabla_{x} w$. Thus
\begin{align*}
\left\langle(-\Delta)^{s} u, v\right\rangle_{L_{\delta}^{2}} & =\left\langle(-\Delta)^{s / 2} u,(-\Delta)^{s / 2} v\right\rangle_{L_{\delta}^{2}}+\left\langle(-\Delta)^{s / 2} u, R_{s-1}^{w} v\right\rangle_{L^{2}} \\
& =\left\langle(-\Delta)^{s / 2} u,(-\Delta)^{s / 2} v\right\rangle_{L_{\delta}^{2}}+ \langle u, (-\Delta)^{s/2}R_{s-1}^w v \rangle_{L^2}.
\end{align*}
For the first term on the right-hand side we compute
$$ \left|\left\langle(-\Delta)^{s / 2} u,(-\Delta)^{s / 2} v\right\rangle_{L_{\delta}^{2}}\right| \leq \|(-\Delta)^{s / 2} u\|_{L_{\delta}^{2}}\|(-\Delta)^{s / 2} v\|_{L_{\delta}^{2}} \leq \|u\|_{H_{\delta}^{s}}\|v\|_{H_{\delta}^{s}} ,$$
by the mapping property $(-\Delta)^{s/2}:H^s_\delta\rightarrow L^2_\delta$. The operator $(-\Delta)^{s / 2} R_{s-1}^{\omega}$ has principal symbol  $-is\nabla_{x} w \cdot \xi|\xi|^{2 s-2}$, and order $2s-1$.
Since $s\geq \frac{1}{2}$, the symbol is bounded at the origin, and it can be treated as a pseudo-differential operator. In particular,
$$
\left|\left\langle u,(-\Delta)^{s / 2} R_{s-1}^{w} v\right\rangle_{L^{2}}\right| \leq\|u\|_{H^{s}}\left\|(-\Delta)^{s / 2} R_{s-1}^{w} v\right\|_{H^{-s}} \leq\|u\|_{H^{s}}\|v\|_{H^{s-1}}.
$$

For the term $\langle P_{2s-1} u, v\rangle_{L^{2}_\delta}$ we compute similarly
$$
\left|\langle P_{2s-1} u, v\rangle_{L^{2}_\delta}\right|=\left|\langle w P_{2 s-1} u, v\rangle_{L^{2}}\right| \leq\|v\|_{H^{s}}\|w P_{2 s-1}u\|_{H^{-s}} \leq\|v\|_{H^{s}}\|u\|_{H^{s-1}},
$$
due to the fact that $wP_{2s-1}$ has order $2s-1, s \geq 1 / 2$, and $w$ depending only on $x$ does not interfere with the computation of the order. For the last term,
$$
|\langle\varepsilon_{r} u, v\rangle_{L_{\delta}^{2}}|=|\langle\varepsilon_{r} w u, v\rangle_{L^{2}}|=\left|\int_{\mathbb{R}^{n}} \varepsilon_{r} w u v d x\right| \leq \|\varepsilon_{r}\|_{L_{\infty}} \int_{\mathbb{R}^{n}}|w u v| \lesssim\|u\|_{L_{\delta}^{2}}\|v\|_{L_{\delta}^{2}}
$$ by Cauchy-Schwarz. Because $H^{s}_\delta\subset H^s\subset H^{s-1}$ and $H^s_\delta\subset L^2_\delta$, we finally obtain
\begin{align*}
    |B(u,v)|&\leq  \left|\left\langle(-\Delta)^{s / 2} u,(-\Delta)^{s / 2} v\right\rangle_{L_{\delta}^{2}}\right| + \left|\left\langle u,(-\Delta)^{s / 2} R_{s-1}^{w} v\right\rangle_{L^{2}}\right| + \left|\langle P_{2s-1} u, v\rangle_{L^{2}_\delta}\right| + k^2|\langle\varepsilon_{r} u, v\rangle_{L_{\delta}^{2}}|
    \\ & \lesssim 
    \|u\|_{H_{\delta}^{s}}\|v\|_{H_{\delta}^{s}} + \|u\|_{H^{s}}\|v\|_{H^{s-1}} + \|v\|_{H^{s}}\|u\|_{H^{s-1}} + \|u\|_{L_{\delta}^{2}}\|v\|_{L_{\delta}^{2}}
    \\ & \lesssim 
    \|u\|_{H_{\delta}^{s}}\|v\|_{H_{\delta}^{s}}.
\end{align*}
\end{proof}

With this in mind, we say that $u\in H^s_\delta$ is a solution of equation \eqref{eq_1} with $F\in (H^s_\delta)^*$ if and only if
$$ B(u,v) = \langle F,v \rangle, \qquad \mbox{for all } v\in H^s_\delta. $$
We are now ready to prove the well-posedness result for equation \eqref{eq_1}.

\begin{prop}[Well-posedness]\label{well-posedness}Let $\delta>0, s\in [\frac{1}{2},1)$, {$\varepsilon_r$ satisfies assumption (H)}, and $F\in (H_\delta^s)^*$. If $$
((-\Delta)^s + P_{2s-1} - k^2\varepsilon_r)u = 0 \quad \textit{in $\R^n$}
$$ has as unique solution $u\equiv 0$ in $H_\delta^s$, then there exists a unique solution $u\in H_\delta^s$  to the equation 
$$
((-\Delta)^s + P_{2s-1} - k^2\varepsilon_r)u = F \quad \textit{in $\R^n$},
$$
and it verifies the estimate
$$ \|u\|_{H^s_\delta}\lesssim \|F\|_{(H^s_\delta)^*}. $$
\end{prop}

\begin{proof}
Let $u,v\in H^s_\delta$. Since $s<1$, by Young's inequality with $\varepsilon$ we see that

$$\|u\|_{H^{s}}\|v\|_{H^{s-1}} \leq\|u\|_{H^{s}}\|v\|_{L^{2}} \leq \frac{C_{\varepsilon}}{2}\|v\|_{L^{2}}^{2}+\frac{\varepsilon}{2}\|u\|_{H^{s}}^{2} \leq \frac{C_{\varepsilon}}{2}\|v\|_{L_{\delta}^{2}}^{2}+\frac{\varepsilon}{2}\|u\|_{H_{\delta}^{s}}^{2}.$$
Thus, by the computations in the proof of Lemma \ref{bilinearwell-defined} we obtain
\begin{align*}
B(u,u)
& \geq\langle(-\Delta)^{s / 2} u,(-\Delta)^{s / 2} u\rangle_{L_{\delta}^{2}}-2\|u\|_{H^{s-1}}\|u\|_{H^{s}}-k^{2}\left\|\varepsilon_{r}\right\|_{L^{\infty}}\|u\|_{L^{2}_\delta}^{2} \\
& \geq\langle(-\Delta)^{s / 2} u,(-\Delta)^{s / 2} u\rangle_{L_{\delta}^{2}}-C_{\varepsilon}\|u\|_{L_{\delta}^{2}}^{2}-\varepsilon\|u\|_{H_{\delta}^{s}}^{2}-k^{2}\left\|\varepsilon_{r}\right\|_{L^{\infty}} \|u\|_{L_{\delta}^{2}}^{2} \\
& =||(-\Delta)^{s / 2} u||_{L_{\delta}^{2}}^{2}-\varepsilon\|u\|_{H_{\delta}^{s}}^{2}-\left(C_{\varepsilon}+k^{2}\left\|\varepsilon_{r}\right\|_{L^\infty}\right)\|u\|_{L_{\delta}^{2}}^{2} \\
& =\left(||(-\Delta)^{s / 2} u||_{L_{\delta}^{2}}^{2}+\|u\|_{L_{\delta}^{2}}^{2}\right)-\varepsilon\|u\|_{H_{\delta}^{s}}^{2}-\left(1+C_{\varepsilon}+k^{2}\left\|\varepsilon_{r}\right\|_{L^\infty}\right) \|u\|_{L^2_\delta}. \\
& =(1-\varepsilon)\|u\|_{H_{\delta}^{s}}^{2}-\left(1+C_{\varepsilon}+k^{2}\left\|\varepsilon_{r}\right\|_{L^{\infty}}\right)\|u\|_{L^{2}_\delta}^{2} .
\end{align*}
Choosing now $\varepsilon \ll 1$ and $l:=1+C_{\varepsilon}+k^{2}\left\|\varepsilon_{r}\right\|_{L^\infty}$, we get
$B(u, u)+l\|u\|^2_{L^2_\delta} \geq C\|u\|^2_{H^s_\delta}$,
which means that the bilinear form $B'(u, v):=B(u, v)+l\langle u, v\rangle_{L^2_\delta}$ is coercive. Moreover, by Lemma \ref{bilinearwell-defined} we have
\begin{align*}
|B'(u,v)| &\leq |B(u,v)| + l|\langle u, v\rangle_{L^2_\delta}|
\lesssim
\|u\|_{H^s_\delta} \|v\|_{H^s_\delta} + \|u\|_{L^2_\delta} \|v\|_{L^2_\delta}
 \lesssim
\|u\|_{H^s_\delta} \|v\|_{H^s_\delta},
\end{align*}
which proves that $B'$ is bounded.



By the Lax-Milgram theorem, there exists a unique $u=G_l\, F$ in $H_\delta^s$ satisfying
\begin{equation*}
  B(u,w)+l\,\langle u,w\rangle_{L_\delta^2}= \langle F,w\rangle \qquad\text{for all }w\in H_\delta^s.
\end{equation*}

The operator $G_l$ is bounded $(H_\delta^s)^{\!*}\to H_\delta^s$. By Proposition \ref{generalize_rellich}, the embedding $H_\delta^s \hookrightarrow L_{\delta'}^2 $ is compact for any $0\leq \delta' <\delta$. Hence   $L_{\delta'}^2 \hookrightarrow (H_\delta^s)^* \to H_\delta^s \hookrightarrow L_{\delta'}^2$ gives a compact, self-adjoint, positive definite operator from $L_{\delta'}^2$ to itself. Now note that
\begin{equation*}
  B\big(v,\cdot\big)-\lambda\,\langle v,\cdot\rangle_{L_{\delta}^2}
  = \langle F,\cdot\rangle\ \text{ on }\  H_{\delta}^s
  \quad\Longleftrightarrow\quad
  v = G_l\big[(l+\lambda)v+ F\big].
\end{equation*}
By the spectral theorem for compact self-adjoint operators, we have that there is a countable set $\Sigma = \{\lambda_j\}_{j=1}^\infty \subset \mathbb{R}$,
$\lambda_1 \le \lambda_2 \le \cdots \to \infty$, such that
if $\lambda \in \mathbb{R} \setminus \Sigma$, then for any 
$F \in \big( H_{\delta}^s)^{*}$, there is a unique 
$u \in  H_{\delta}^s$ satisfying
\begin{equation*}
  B(u,w) - \lambda\,(u,w)_{L_{\delta}^2} = (F,w)_{L_{\delta}^2}
  \quad \text{for all } w \in H_{\delta}^s.
\end{equation*}
In particular, because in our case by assumption we have $0\notin \Sigma$, then there exists a unique weak solution $u\in H_{\delta}^s$ to the equation $$((-\Delta)^s + P_{2s-1} - k^2\varepsilon_r)u = F \quad \mbox{ in } \R^n.$$
\end{proof}

Having proved the well-posedness of Proposition \ref{well-posedness}, we can now prove our main Theorem \ref{theorem1.1}.

\begin{proof}[Proof of Theorem \ref{theorem1.1}]
In light of the equivalence Lemma \ref{equivalence-lemma}, it is enough to prove the well-posedness of equation \eqref{EsEi}. We show that the right hand side of \eqref{EsEi} $k^2(\varepsilon_r-1)\widetilde{E}_i- P_{2s-1} \widetilde E_i$ is in $(H_\delta^s)^*$ for $E_i = \mathcal{C}^*\nabla \times (pe^{ik^{1/s}x\cdot d})$. As observed in Lemma \ref{Pi_Ei}, we have $\sigma\widetilde E_i \in L^2(\mathbb R^n)$ whenever $\sigma\in C^\infty_c(\mathbb R^n)$. In particular, this implies that $k^2(\varepsilon_r-1)\widetilde E_i$ and $\nabla (\log \varepsilon_r)\cdot \widetilde E_i$ both belong to $L^2(\mathbb R^n)$, because $\varepsilon_r \equiv 1$ in $\mathbb R^n\setminus\Omega$. Thus 
\begin{align*}
    k^2(\varepsilon_r-1)\widetilde{E}_i- P_{2s-1} \widetilde E_i = k^2(\varepsilon_r-1)\widetilde{E}_i+ (-\Delta)^{s-1}\nabla( \nabla (\log \varepsilon_r)\cdot \widetilde E_i) \in H^{1-2s}(\mathbb R^n),
\end{align*}
given that $s\geq 1/2$. Because $H^{1-2s}\subset H^{-s}= \left(H^s\right)^* \subset\left(H_\delta^s\right)^*$, we see that equation \eqref{EsEi} is well-posed in $H^s_\delta$ by Proposition \ref{well-posedness}.
\end{proof}

\subsection{Well-posedness for two-point fields}
Finally, we show that the well-posedness of equation \eqref{eq:E_s} in the sense of Definition \ref{def:well-posedness-2-points} can be deduced from the well-posedness of equation \eqref{EsEi}, thus proving Theorem \ref{theorem1.2}.

\begin{proof}[Proof of Theorem \ref{theorem1.2}]
    Let $p,d\in \mathbb S^{n-1}\setminus\{0\}$, and assume that the incoming electric field is $E_i:= \mathcal C^*\nabla\times(p e^{ik^{1/s}x\cdot d})$. By Theorem \ref{theorem1.1}, there exists a unique solution $\widetilde E_s\in H^s_\delta\subset H^s$ to equation \eqref{EsEi}. By the properties of the projection operator $\Pi$ from Proposition \ref{Bijectivity of Pi}, there exists a unique $E_s\in H^s_{fHd}(\mathbb R^{2n})$ such that $\widetilde E_s = \Pi E_s$. Thus equation \eqref{eq:E_s} has a unique solution in the sense of Definition \ref{def:well-posedness-2-points}.
\end{proof}

\section{Future work}

{The inverse scattering problem consists in determining $\varepsilon$ and $\mu$ inside the scatterer from measurements of the scattered field $(\mathbf{E},\mathbf{H})$ at infinity, for one or several incident fields. Classical inverse scattering for Maxwell equations has been extensively studied. The foundational framework was developed in \cite{ColtonKress1998}. Global uniqueness for sufficiently smooth electromagnetic parameters was established in \cite{OPS93}, and later extended in \cite{KSU11} to an admissible Riemannian manifold, and in Euclidean space with admissible matrix coefficients. In the context of anisotropic media, the role of transmission eigenvalues and their relation to far-field data was analyzed in \cite{CCMS10}.

Fractional analogues of inverse scattering have recently been studied. In \cite{UW25}, uniqueness from fixed-energy scattering data was proved for the relativistic Schr\"odinger operator $(-\Delta)^{1/2}+V$. In \cite{DGM25}, high-energy inverse scattering results were obtained for $(-\Delta)^s+V$ for some values of $s$ (not including $s = \frac{1}{2}$) and for high energies.

As future work, we will consider inverse scattering for the fractional Maxwell equations. In particular, given an exterior scattering measurement, such as the scattering amplitude, we want to determine the material parameters $\varepsilon$ and $\mu$. This also motivates our study of the forward problem in the scattering setting.
}

\bigskip

\bibliographystyle{plain}
\bibliography{references}


\bigskip

\newpage

\end{document}